\documentclass[10pt]{article}
\usepackage{amssymb, amsmath, amsfonts,latexsym, oldgerm,amscd,amsthm}
\usepackage{relsize}
\usepackage{url}
\usepackage{amsbsy}
\usepackage{mathtools,mathdots} 
\usepackage{pdfsync}
\usepackage{enumitem}
\usepackage{graphicx}
\usepackage{ytableau}

\textwidth13cm \usepackage[all]{xy} 
\newcommand{\tn}{\textnormal} \newcommand{\lra}{\longrightarrow}
 \newtheorem{de}{Definition}[section]

\newtheorem{thm}[de]{Theorem} \newtheorem{lem}[de]{Lemma}
 \newtheorem{nt}[de]{Notation}
 
 \newtheorem{rmk}[de]{Remark}

\def\E{{\rm E}}

\def\ESp{{\rm ESp}}
\def\Sp{{\rm Sp}}

\def\GL{{\rm GL}}

\def\SL{{\rm SL}}

\def\Um{{\rm Um}}

\def\End{{\rm End}}
\def\Aut{{\rm Aut}}
\def\Trans{{\rm Trans}}
\def\ETrans{{\rm ETrans}}

\def\lra{\longrightarrow}

\newcommand{\gm}{\mathfrak{m}}

\begin{document}

\title{An analogue of a result of Tits for linear and symplectic transvection groups}

\author{
Pratyusha Chattopadhyay\\
{\small Statistics and Mathematics Unit, Indian Statistical Institute, Bangalore 560 059, India}
}

\date{}

\maketitle


\begin{center}
\end{center}

{\small ~~~~Abstract: In \cite{N} Bogdan Nica presented an elementary proof of a result which says that the 
relative elementary linear group with respect to a square of an ideal of the ring is a subset of the true relative elementary linear group. The original result was proved by J. Tits in \cite{T} in the much general context of Chevalley groups. In this paper we prove analogues of the result of Tits for linear transvection 
group and symplectic transvection group. We also obtain an elementary proof of a special case of Tits's result, 
namely the case of elementary symplectic group, using commutator identities for generators of this group.
}

\section{\large Introduction}

Let $R$ be a commutative ring with $1$. Let E$_{ij}(\lambda) : = I_{n} + e_{ij}(\lambda)$, $1 \leq i \neq j \leq n$, 
$\lambda \in R$, where $e_{ij}(\lambda) \in$ M$_{n}(R)$ has an entry $\lambda$ in its $(i, j)$-th position and 
zeros elsewhere. $\E_n(R)$ is the subgroup of $\SL_n(R)$ generated as a group by the elements {\rm E}$_{ij}(a)$, $a \in R$, $1 \leq i \neq j \leq n$. For an ideal $I$ of $R$ the relative elementary linear group {\rm E}$_n(I)$ is the subgroup of {\rm E}$_n(R)$ generated as a group by the elements {\rm E}$_{ij}(x)$, $x \in I$, $1 \leq i \neq j \leq n$. The relative elementary linear group ${\rm E}_n(R, I)$ is defined as the normal closure of {\rm E}$_n(I)$ in {\rm E}$_n(R)$.  

A well known normality theorem of  A.A. Suslin says that the elementary linear group $\E_n(R)$ is normal in the 
general linear group $\GL_n(R)$, for $n \ge 3$. Suslin also proved a stronger relative version of this theorem 
which says that the relative elementary linear group $\E_n(R, I)$, where $I$ is an ideal of a commutative ring
$R$ with 1, is normal in the general linear group $\GL_n(R)$, for $n \ge 3$ (see \cite{Su1}, Corollary 1.4). 
We note that G. Taddei extended Suslin's normality theorem to the set up of Chevalley groups over rings in \cite{Ta}.

In \cite{N} B. Nica while proving  the so called "true relative of Suslin's normality theorem" also gave an elementary proof of the fact that  the group ${\rm E}_n(R, I^2)$ is a subgroup of ${\rm E}_n(I)$ (see \cite{N}, Theorem 3). This result is a special case of a result of J. Tits (see \cite{T}, Proposition 3) which was originally proved in the much general context of Chevalley groups. It should be mentioned that A. Stepanov reproved Tits's result for elementary groups in \S 3 of \cite{St}.

H. Bass introduced a very important class of groups, called
Transvection groups, in the study of projective modules and  their
K-theory; defnitions of these groups are recalled in \S 4 and \S 5.
We also refer to  \cite{ho} for these basic definitions. Our main goal in this paper is to prove analogues of the above mentioned result of J. Tits for linear transvection groups (Theorem \ref{Tits-linTrans}) and symplectic transvection groups (Theorem \ref{Tits-sympTrans}). As in Theorem 2, \cite{N}  we also obtain an elementary proof, different from that in \cite{St}, of a special case 
of Tits's result, namely the case of elementary symplectic group $\ESp_{2n}(R)$, using commutator identities for generators of $\ESp_{2n}(R)$ (Theorem \ref{tits-symp}).

\section{\large Preliminaries}

Here we recall some basic notions which will be used in the subsequent sections.

\medskip
In this article we will always assume that $R$ is a commutative ring with $1$. 
A column ${ v} = (v_{1}, \ldots, v_{n})^t \in R^{n}$ is said to be {\it
unimodular} if there are elements $w_{1}, \ldots, w_{n}$ in $R$ such
that $v_{1}w_{1} + \cdots + v_{n}w_{n} = 1$. Um$_{n}(R)$ will denote
the set of all unimodular columns ${ v} \in R^{n}$. Let $I$ be an ideal
in $R$. We denote by ${\rm Um}_n(R, I)$ the set of all unimodular columns
of length $n$ which are congruent to $e_1 = (1, 0, \ldots, 0)$ modulo
$I$. (If $I = R$, then ${\rm Um}_n(R, I)$ is ${\rm Um}_n(R)$).

\begin{de} {\rm Let $P$ be a finitely generated projective $R$-module. 
An element $p \in P$ is said to be {\it unimodular} if there exists a 
$R$-linear map $\phi: P \to R$ such that $\phi(p)=1$. The collection 
of all unimodular elements of $P$ is denoted by $\Um(P)$. 

Let $P$ be of the form $R \oplus Q$ and have an element of the form
$(1,0)$ which correspond to the unimodular element. An element $(a,q)
\in P$ is said to be {\it relative unimodular} with respect to an ideal $I$ of 
$R$ if $(a,q)$ is unimodular and $(a,q)$ is congruent to $(1,0)$
modulo $IP$.  The collection of all relative unimodular elements
with respect to an ideal $I$ is denoted by ${\rm Um}(P,IP)$.  }
\end{de}

Let us recall that if $M$ is a finitely presented $R$-module and $S$
is a multiplicative set of $R$, then $S^{-1} {\rm Hom}_R(M,R) \cong
{\rm Hom}_{R_S}(M_S, R_S)$ (Theorem 2.13", Chapter I, \cite{L}). Also 
recall that if $f=(f_1, \ldots, f_n)\in R^n := M$, then $\Theta_M(f)=\{ \phi(f): 
\phi \in {\rm  Hom}(M,R) \}= \sum_{i=1}^n Rf_i$. Therefore, if $P$ is a finitely
generated projective $R$-module of rank $n$, $\gm$ is a maximal ideal
of $R$ and $v\in \Um(P)$, then $v_\gm \in \Um_n(R_\gm)$. Similarly if
$v \in \Um(P,IP)$ then $v_\gm \in \Um_n(R_\gm, I_\gm)$.

\begin{de} {\bf Elementary Linear Group:} 
{\rm The elementary linear group E$_{n}(R)$ denote the subgroup of SL$_{n}(R)$ 
consisting of all {\it elementary} matrices, i.e. those matrices which are a finite
product of the {\it elementary generators} E$_{ij}(\lambda) = I_{n} +
e_{ij}(\lambda)$, $1 \leq i \neq j \leq n$, $\lambda \in R$, where
$e_{ij}(\lambda) \in$ M$_{n}(R)$ has an entry $\lambda$ in its $(i,
j)$-th position and zeros elsewhere.}
\end{de}

In the sequel, if $\alpha$ denotes an $m \times n$ matrix, then we let
$\alpha^t$ denote its {\it transpose} matrix. This is of course an
$n\times m$ matrix. However, we will mostly be working with square
matrices, or rows and columns.

\begin{de} {\bf The Relative Groups E$_n(I)$, E$_n(R,I)$:} {\rm Let $I$ be
  an ideal of $R$. The relative elementary linear group {\rm E}$_n(I)$ is the
  subgroup of {\rm E}$_n(R)$ generated as a group by the elements {\rm
    E}$_{ij}(x)$, $x \in I$, $1 \leq i \neq j \leq n$. This group is also known as
   {\it  true relative elementary linear group}.

  The relative elementary linear group ${\rm E}_n(R, I)$ is the normal
  closure of {\rm E}$_n(I)$ in {\rm E}$_n(R)$. (Equivalently, ${\rm E}_n(R, I)$ 
  is generated as a group by ${\rm E}_{ij}(a) {\rm E}_{ji}(x)$E$_{ij}(-a)$,\tn{ with} 
  $a \in R$, $x \in I$, $i \neq j$, provided $n \geq 3$ (see \cite{V3}, Lemma 8)).}
\end{de}

\begin{de}
{\rm $\E_n^1(I)$ is the subgroup of $\E_n(R)$ generated by the elements 
$E_{1i}(x)$ and $E_{j1}(y)$, where $x, y \in I$, and $2 \le i, j \le n$.}
\end{de}

\begin{rmk} \label{inclusion-linear}
{\rm Note that the generators for the elementary linear group $\E_n(R)$ 
satisfies the commutator relation $[ e_{ij}(a), e_{jk}(b)] = e_{ik}(ab)$.
Using this relation we can show that $\E_n(I^2) \subseteq \E_n^1(I)$.}
\end{rmk}

\begin{de} {\bf Symplectic Group Sp$_{2n}(R)$:} {\rm The symplectic group ${\rm
  Sp}_{2n}(R)=\{\alpha \in {\rm GL}_{2n}(R)\,\,|\,\, \alpha^t \psi_n
  \alpha = \psi_n \}$, where $\psi_n = \underset{i=1}{\overset{n}\sum}
  e_{2i-1,2i}- \underset{i=1}{\overset{n}\sum} e_{2i,2i-1}$, the
  standard symplectic form.}
\end{de}

\tn{Let $\sigma$ denote the permutation of the natural numbers $\{ 1, 2, \ldots, 2n \}$ 
given by $\sigma(2i)=2i-1$ and $\sigma(2i-1)=2i$}.

\begin{de} \label{2.4} {\bf Elementary Symplectic Group
    ESp$_{2n}(R)$:} {\rm We define for $z \in R$, $1\le i\ne j\le 2n$,

\[ se_{ij}(z) =
\begin{cases}
  1_{2n}+ze_{ij}  & \mbox{{\rm if}  $i=\sigma(j)$,} \\
  1_{2n}+ze_{ij}-(-1)^{i+j}z e_{\sigma (j) \sigma (i)} & \mbox{{\rm
      if} $i\ne \sigma(j)$.}
\end{cases}
\]

It is easy to check that all these elements belong to {\rm Sp}$_{2n}(R)$. 
We call them {\it elementary symplectic matrices} over $R$ and the
subgroup of {\rm Sp}$_{2n}(R)$ generated by them is called the
{\it elementary symplectic group} {\rm ESp}$_{2n}(R)$.}
\end{de}

\begin{de} {\bf The Relative Group $\ESp_{2n}(I), \ESp_{2n}(R,I)$:}
{\rm Let $I$ be an ideal of $R$. The relative elementary symplectic 
group $\ESp_{2n}(I)$ is the subgroup $\ESp_{2n}(R)$ generated by 
the elements $se_{i j}(x), x \in I$ and $1 \le i \ne j \le 2n$. This group is also known as
    {\it true relative elementary symplectic group}.

The relative elementary symplectic group $\ESp_{2n}(R, I)$ is the normal closure 
of $\ESp_{2n}(I)$ in $\ESp_{2n}(R)$.}
\end{de}

\begin{de}
  {\rm The group $\ESp_{2n}^1(I)$ is the subgroup of $\ESp_{2n}(R)$
  generated by the elements of the form $se_{1 i}(x)$ and $se_{j
    1}(y)$, where $x, y \in I$ and $2 \le i, j \le 2n$.}
\end{de}

\begin{rmk}  \label{inclusion-symplectic}
{\rm Note that the generators of elementary symplectic group $\ESp_{2n}(R)$ 
satisfies commutator relations $[se_{ik}(a), se_{kj}(b)] ~=~
se_{ij}(ab)$, where $i \ne \sigma(j), k \ne \sigma(i), \sigma(j)$ and $[se_{ik}(a),
  se_{k \sigma(i)}(b)] ~=~ se_{i \sigma(i)}(2ab)$, where $k \ne i, \sigma(i)$. Using these
  relations we can show that when $R=2R$, we have
$\ESp_{2n}(I^2) \subseteq \ESp_{2n}^1(I)$.}
\end{rmk}

\begin{nt}
{\rm Let $M$ be a finitely presented $R$-module and $a$ be a
  non-nilpotent element of $R$. Let $R_a$ denote the ring $R$
  localised at the multiplicative set $\{a^i : i \ge 0 \}$ and $M_a$
  denote the $R_a$-module $M$ localised at $\{a^i : i \ge 0 \}$. Let
  $\alpha(X)$ be an element of $\End(M[X])$. The localization map $i:
  M \to M_a$ induces a map $i^*: \End(M[X]) \to \End(M[X]_a)
  = \End(M_a[X])$. We shall denote $i^*(\alpha(X))$ by $\alpha(X)_a$
  in the sequel.}
\end{nt}

\section{\large The free case}

In this section we obtain an elementary proof of Tits's result for the elementary symplectic groups. 
We begin by recalling a result of B. Nica.

\begin{thm} \label{tits-lin}
Let $R$ be a ring and $I$ be an ideal of $R$. Then $\E_n(R, I^2) \subseteq \E_n(I)$, for
$n \ge 3$.
\end{thm}

\begin{proof} 
See proof of Theorem 3 in \cite{N}.
\end{proof}

Following ideas from \cite{K} we will establish a symplectic analogue of the above result. We 
first state and prove a few lemmas. We use the notation $\widetilde{v} = v^t \psi_n$, for a row 
vector $v^t$ of length $2n$ and the standard symplectic form $\psi_n$.

\begin{lem} \label{short-rt}
Let $R$ be a ring with $R=2R$ and $I$ be an ideal of $R$. For $n \ge 2$, $v \in R^{2n}$, and $a,b \in I$ we have 
$I_{2n+2} + ab v \widetilde{v} \in [\ESp_{2n+2}(I), \ESp_{2n+2}(I)]$.
\end{lem}

\begin{proof}
The result follows from the identity below.
\begin{align*}
\begin{pmatrix}I_{2n} + ab v \widetilde{v} & 0 & 0 \\ 0 & 1 & 0 \\ 0 & 0 & 1 \end{pmatrix} =
\begin{bmatrix}{\begin{pmatrix}I_{2n} & 0 & (a/2) v \\ (a/2) \widetilde{v} & 1 & 0 \\ 0 & 0 & 1\end{pmatrix}},
{\begin{pmatrix}I_{2n} & -bv & 0 \\ 0 & 1 & 0 \\ b \widetilde{v} & 0 & 1\end{pmatrix}}\end{bmatrix} \\[-\normalbaselineskip]\tag*{\qedhere}
\end{align*}
\end{proof}

\begin{lem} \label{long-rt}
Let $R$ be a ring and $I$ be an ideal of $R$. For $n \ge 2$, $v, w \in R^{2n}$ with $\widetilde{w} v = 0$, and 
$a, b \in I$ we have $I_{2n+2} + ab v \widetilde{w} + ab w \widetilde{v} \in [\ESp_{2n+2}(I), \ESp_{2n+2}(I)]$.
\end{lem}

\begin{proof}
The result follows from the identity below.
\begin{align*}
\begin{pmatrix} I_{2n} + ab v \widetilde{w} + ab w \widetilde{v} & 0 & 0 \\ 0 & 1 & 0 \\ 0 & 0 & 1 \end{pmatrix} 
=
\begin{bmatrix} \begin{pmatrix} I_{2n} & av & 0 \\ 0 & 1 & 0 \\ -a \widetilde{v} & 0 & 1 \end{pmatrix} , \begin{pmatrix} I_{2n} & 0 & bw \\ b \widetilde{w} & 1 & 0 \\ 0 & 0 & 1 \end{pmatrix} \end{bmatrix} \\[-\normalbaselineskip]\tag*{\qedhere}
\end{align*}
\end{proof}

\begin{lem} \label{long-rt-sp}
Let $R$ be a ring with $R=2R$ and $I$ be an ideal of $R$. Also,
assume $n \ge 2$, $v,w \in R^{2n}$ with $\widetilde{w} v = 0$, and $a, b \in I$. If $v_i = v_{\sigma(i)} = 0$, 
then $I_{2n} + ab v \widetilde{w} + ab w \widetilde{v} \in \ESp_{2n}(I)$.
\end{lem}

\begin{proof}
For the sake of definiteness we assume $v_{2n-1} = v_{2n} = 0$. Let us denote $v = (v', 0, 0)^t$ and $w = (w', x, y)^t$, where $x, y \in R$. Now,
\begin{eqnarray*}
&& I_{2n} + ab v \widetilde{w} + ab w \widetilde{v} \\
&=& \begin{pmatrix} I_{2n-2} + ab v' \widetilde{w'} + ab w' \widetilde{v'} & 0 & 0 \\ 0 & 1 & 0 \\ 0 & 0 & 1 \end{pmatrix}  \begin{pmatrix} I_{2n-2} & -abyv' & 0 \\ 0 & 1 & 0 \\ aby \widetilde{v'} & 0 & 1\end{pmatrix}  \begin{pmatrix} I_{2n-2} & 0 & abxv' \\ abx \widetilde{v'} & 1 & 0 \\ 0 & 0 & 1\end{pmatrix}  \\
 && \begin{pmatrix} I_{2n-2} + (ab)^2 xy v' \widetilde{v'} & 0 & 0 \\ 0 & 1 & 0 \\ 0 & 0 & 1 \end{pmatrix}.
\end{eqnarray*}

Clearly the second and the third factor belong to $\ESp_{2n}(I)$. The first factor belongs to $\ESp_{2n}(I)$ due to Lemma \ref{long-rt} and the fourth factor belong to $\ESp_{2n}(I)$ due to Lemma \ref{short-rt}. Hence the result follows.
\end{proof}

\begin{lem} \label{short-rt-sp}
Let $R$ be a ring with $R=2R$ and I be an ideal of $R$. For $n \ge 2$, $v \in R^{2n}$, and $a, b \in I$ we have 
$I_{2n} + ab v \widetilde{v} \in \ESp_{2n}(I)$.
\end{lem}

\begin{proof}
Let us denote $v = v' + v''$, where $v' = (v_1, v_2, \ldots, v_{2n-2}, 0, 0)$ and $v'' = (0, 0, \ldots, 0, v_{2n-1}, v_{2n})$. Note that 
\begin{eqnarray*}
I_{2n} + ab v \widetilde{v} &=& (I_{2n} + ab v' \widetilde{v'}) ~ (I_{2n} + ab v'' \widetilde{v'} + ab v' \widetilde{v''}) ~ (I_{2n} + ab v'' \widetilde{v''}).
\end{eqnarray*}

Here the first and the third factor belong to $\ESp_{2n}(I)$ due to Lemma \ref{short-rt} and the second factor belongs to $\ESp_{2n}(I)$ due to Lemma \ref{long-rt-sp}.  Hence the result follows.
\end{proof}

\begin{rmk} \label{kernel-Um}
{\rm Let $w \in I^n (\subseteq R^n)$ and $v \in \Um_n(R)$ be such that $w^t v = 0$. Then $w = \sum_{i < j} a_{ij}(v_j e_i - v_i e_j)$, for $a_{ij} \in I$. Here $e_i \in R^n$ is a column vector of length $n$ which has $1$ at the i-th place and zeros elsewhere. This is a well known result. For proof one can see Chapter 5 of \cite{M} or \cite{sv}.}
\end{rmk}

Following identity plays a crucial role in our  proof of Tits's result for elementary symplectic group.

\begin{lem} \label{imp}
Let $w \in R^n$ and $u_i \in I^n (\subseteq R^n)$ for $i = 1, 2, \ldots, k$ be such that $\widetilde{u_i} w = 0$ for all $i$. Then there is an $x \in I$ such that $I_n + \sum_{i=1}^k (u_i \widetilde{w} + w \widetilde{u_i}) = \prod_{i=1}^k (I_n + u_i \widetilde{w} + w \widetilde{u_i}) ~ (I_n + x w \widetilde{w})$.
\end{lem}
 
\begin{proof}
This can be proved using induction on $k$. For details one can see proof of Lemma 1.9 in \cite{K}.
\end{proof}

\begin{lem} \label{long-rt-sp(2)}
Let $R$ be a ring with $R=2R$ and $I$ be an ideal of $R$. Also, assume $n \ge 3$, $v \in R^{2n}$, $w \in \Um_{2n}(R)$ with $\widetilde{v} w = 0$, and $a, b \in I$. Then $\gamma = I_{2n} + ab v \widetilde{w} + ab w \widetilde{v} \in \ESp_{2n}(I)$.
\end{lem}

\begin{proof}
By Remark \ref{kernel-Um}, $\widetilde{v}$ can be written in the form $\sum_{i < j} x_{ij} (w_i e_j - w_j e_i)$, where $x_{ij} \in R$. Note that $v = \psi_n \widetilde{v}^t = \sum_{i <j} x_{ij} ((-1)^j w_i e_{\sigma(j)} - (-1)^i w_j e_{\sigma(i)})$. Then $\gamma$ can be written as $I_{2n} + \sum_{i < j} \big(  ab x_{ij} ((-1)^j w_i e_{\sigma(j)} - (-1)^i w_i e_{\sigma(i)})^t \widetilde{w} + ab x_{ij} w (w_i e_j - w_j e_i)    \big)$. Hence by Lemma \ref{imp}, $\gamma$ can be written as the following product 
\begin{eqnarray*}
\prod_{i < j} \big(  I_{2n} +  ab x_{ij} ((-1)^j w_i e_{\sigma(j)} - (-1)^i w_i e_{\sigma(j)})^t \widetilde{w} + ab x_{ij} w (w_i e_j - w_j e_i)  \big)  \big( I_{2n} + y w \widetilde{w} \big).
\end{eqnarray*}
where $y \in I$. By Lemma \ref{short-rt-sp}, $\big( I_{2n} + y w \widetilde{w} \big) \in \ESp_{2n}(I)$ and by Lemma \ref{long-rt-sp}, 
$\big(  I_{2n} +  x_{ij} ((-1)^j w_i e_{\sigma(j)} - (-1)^i w_i e_{\sigma(j)})^t \widetilde{w} + x_{ij} w (w_i e_j - w_j e_i)  \big) \in \ESp_{2n}(I)$, for $n \ge 3$. Therefore, $\gamma \in \ESp_{2n}(I)$.
\end{proof}

\begin{thm} \label{tits-symp}
Let $R$ be a ring with $R=2R$ and $I$ be an ideal of $R$. Then $\ESp_{2n}(R, I^2) \subseteq \ESp_{2n}(I)$, for
$n \ge 3$.
\end{thm}

\begin{proof} 
An element of $\ESp_{2n}(R, I^2)$ is product of elements of the form $g se_{ij}(ab) g^{-1}$, where $g \in \ESp_{2n}(R)$ and $a,b \in I$. Here either $i = \sigma(j)$ or $i \ne \sigma(j)$.

When $i = \sigma(j)$, we have
\begin{eqnarray*}
g ~ se_{ij}(ab) ~ g^{-1} = I_{2n} + ab v \widetilde{v},
\end{eqnarray*}
where $v$ is the $i$-th column of $g$. Hence by Lemma \ref{short-rt-sp}, $g ~ se_{ij}(ab) ~ g^{-1} \in \ESp_{2n}(I)$.

When $i \ne \sigma(j)$, we have
\begin{eqnarray*}
g ~ se_{ij}(ab) ~ g^{-1} = I_{2n} + ab v \widetilde{w} + ab w \widetilde{v},
\end{eqnarray*}
where $v$ is the $i$-th column of $g$ and $w$ is the $\sigma(j)$-th column of $g$. Hence by Lemma \ref{long-rt-sp(2)},
$g ~ se_{ij}(ab) ~ g^{-1} \in \ESp_{2n}(I)$.
\end{proof}

\section{\large{The case of linear transvections groups}}

In this section we deal with the linear transvection group, and prove an analogue of Tits's result in this set-up.
We begin by recalling some basic notions. Following H. Bass one  defines a linear transvection of a finitely generated $R$-module as follows:

\begin{de}
{\rm Let $M$ be a finitely generated $R$-module. Let $q \in M$ and
  $\pi \in M^*= {\rm Hom}(M,R)$, with $\pi(q) = 0$. Let $\pi_q(p) :=
  \pi(p)q$. An automorphism of the form $1+\pi_q$ is called a {\it
   linear transvection} of $M$, if either $q \in \Um(M)$ or $\pi \in
  \Um(M^*)$. The collection of all linear transvections of $M$ is denoted by
  $\Trans(M)$. This forms a subgroup of $\Aut(M)$.}
\end{de}

\begin{de}
{\rm Let $M$ be a finitely generated $R$ module. Automorphisms of
  $N= (R \oplus M)$ of the form
\begin{eqnarray*}
(a, p)^t & \mapsto & (a, p+ax)^t,
\end{eqnarray*}
or of the form
\begin{eqnarray*}
(a, p)^t & \mapsto & (a+ \tau(p), p)^t,
\end{eqnarray*}
where $x \in M$ and $\tau \in M^*$ are called {\it elementary
  linear transvections} of $N$. Let us denote the first automorphism by $E(x)$
and the second one by $E^*(\tau)$.  It can be verified that these are
transvections of $N$.  Indeed, let us consider $\pi(t,y)=t$, $q=(0,x)$ to
get $E(x)$, and consider $\pi(a,p)= \tau(p)$, where $\tau \in
M^*$, $q=(1,0)$ to get $E^*(\tau)$. The subgroup of $\Trans(N)$
generated by elementary transvections is denoted by $\ETrans(N)$.}
\end{de}

\begin{de} {\rm Let $I$ be an ideal of $R$. Elementary
    transvections of $N=(R \oplus M)$ of the form $E(x), E^*(\tau)$,
    where $x \in IM$ and $\tau \in IM^*$ are called {\it relative
      elementary linear transvections} with respect to an ideal $I$, and the
    group generated by them is denoted by $\ETrans(IN)$. The normal
    closure of $\ETrans(IN)$ in $\ETrans(N)$ is denoted by
    $\ETrans(N,IN)$.}
\end{de}

\begin{lem} \label{linear,free,true-rel} Let $I$ be an ideal of $R$, $M$
  be a free $R$ module of rank $n \ge 2$, and $N=(R \oplus M)$. Then
$\ETrans(IN) = \E_{n+1}^1(I)$.
\end{lem}

\begin{proof}
In this case $\ETrans(IN)$ is generated by the elements of the form
$\big( \begin{smallmatrix} 1 & 0 \\ x & I_n \end{smallmatrix} \big)$, and 
$\big( \begin{smallmatrix} 1 & y^t  \\ 0 & I_n \end{smallmatrix} \big)$,
where $x, y \in I^n (\subseteq R^n)$. Hence, $\ETrans(IN) \subseteq
\E_{n+1}^1(I)$. Note that generators of $\E_{n+1}^1(I)$ 
belongs $\ETrans(IN)$, when $N$ is free, and hence $\E_{n+1}^1(I) 
\subseteq \ETrans(IN)$.
\end{proof}

\begin{lem} \label{linear,free,rel} Let $I$ be an ideal of $R$ and $M$
  be a free $R$ module of rank $n \ge 2$, and $N=(R \oplus M)$. Then
$\ETrans(N,IN) = \E_{n+1}(R,I)$.
\end{lem}

\begin{proof}
See proof of Lemma 4.5 in \cite{cr2}.
\end{proof}

\begin{lem} \label{commE(I)} Let $R$ be a ring, $I$ be an ideal of $R$,
and $n \ge 3$. Let $\varepsilon =
  \varepsilon_1 \ldots \varepsilon_r$,
  where each $\varepsilon_k$ is an elementary generator of $\E_n^1(I)$.  If
  $E_{ij}(a(X)) \in \E_n^1(I[X])$, then
\begin{eqnarray*}
\varepsilon ~ E_{ij}(Y^{4^r} a(X)) ~ \varepsilon^{-1} 
&=& \prod_{t=1}^s E_{i_t j_t}(Y b_t(X,Y)),
\end{eqnarray*}
where $b_t (X,Y) \in I[X, Y]$, and either $i_t =1$ or $j_t =1$.
\end{lem}

\begin{proof} 
We prove the result using induction on $r$, where $\varepsilon$ is
product of $r$ many elementary generators. Let $r=1$ and $\varepsilon
= E_{pq}(c)$, where $c \in I$. Also, for the elementary generator $E_{ij}(Y^{4^r} a(X))$
we assume $i=1$.

{\it Case $($1$)$:} Let $(p,q)=(1,j)$. In this case we get that 
\begin{eqnarray*}
E_{1j}(c) ~ E_{1j}(Y^4 a(X)) ~ E_{1j}(-c) &=& E_{1j}(Y^4 a(X)).
\end{eqnarray*}

{\it Case $($2$)$:} Let $(p,q) = (1,k)$ and $k \ne j$. In this case we get that 
\begin{eqnarray*}
E_{1k}(c) ~ E_{1j}(Y^4 a(X)) ~ E_{1k}(-c) &=& E_{1j}(Y^4 a(X)).
\end{eqnarray*}

{\it Case $($3$)$:} Let $(p,q)=(j, 1)$. In this case we get that 
\begin{eqnarray*}
&& E_{j1}(c) ~ E_{1j}(Y^4 a(X)) ~ E_{j1}(-c) \\
&=& E_{j1}(c) ~ [E_{1k}(Y^2 a(X)), E_{kj}(Y^2)] E_{j1}(-c) \\
&=& [E_{jk}(Y^2 ca(X)) ~ E_{1k}(Y^2 a(X)), ~ E_{k1}(-Y^2c) ~ E_{kj}(Y^2)] \\
&=& E_{jk}(Y^2 ca(X)) ~ E_{1k}(Y^2 a(X)) ~ E_{k1}(-Y^2c) ~ E_{kj}(Y^2) ~ E_{1k}(-Y^2 a(X))\\
&& E_{jk}(-Y^2 ca(X)) ~ E_{kj}(-Y^2) ~ E_{k1}(Y^2c) \\
&=& E_{jk}(Y^2 ca(X)) ~ E_{1k}(Y^2 a(X)) ~ E_{k1}(-Y^2c) ~ E_{1k}(-Y^2 a(X)) ~ E_{1k}(Y^2 a(X)) \\
&& E_{kj}(Y^2) ~ E_{1k}(-Y^2 a(X))  ~ E_{kj}(-Y^2) ~ E_{kj}(Y^2) ~ E_{jk}(-Y^2 ca(X)) \\
&& E_{kj}(-Y^2) ~ E_{k1}(Y^2c) \\
&=& E_{jk}(Y^2 ca(X)) ~ E_{1k}(Y^2 a(X)) ~ E_{k1}(-Y^2c) ~ E_{1k}(-Y^2 a(X)) ~ E_{1j}(Y^4 a(X)) \\
&& E_{kj}(Y^2) ~ [E_{j1}(-Y a(X)), ~ E_{1k}(Yc)] ~ E_{kj}(-Y^2) ~ E_{k1}(Y^2x) \\
&=& [E_{j1}(Y a(X)), ~E_{1k}(Yc)] ~ E_{1k}(Y^2 a(X)) ~ E_{k1}(-Y^2c) ~ E_{1k}(-Y^2 a(X)) \\
&& E_{1j}(Y^4 a(X)) ~ [E_{k1}(-Y^3 a(X)) ~ E_{j1}(-Y a(X)), ~ E_{1j}(-Y^3 c) ~ E_{1k}(Yc)]  \\
&& E_{k1}(Y^2c). 
\end{eqnarray*}

{\it Case $($4$)$:} Let $(p,q)=(k,1)$, where $k \ne j$. In this case we get that
\begin{eqnarray*}
&& E_{k1}(c) ~ E_{1j}(Y^4 a(X)) ~ E_{k1}(-c) \\
&=& E_{kj}(Y^4 c a(X)) ~ E_{1j}(Y^4 a(X))\\
&=& [ E_{k1}(Y^2 a(X)), ~ E_{1j}(Y^2c)] ~ E_{1j}(Y^4 a(X)).
\end{eqnarray*}

Hence the result is true when $i=1$ and $\varepsilon$ is an elementary
generator. Carrying out similar computations one can show that the result is true
when $j=1$ and $\varepsilon$ is an elementary generator.
Let us assume that the result is true when $\varepsilon$ is product of $r-1$
many elementary generators, i.e, $\varepsilon_2 \ldots \varepsilon_r ~
E_{ij}(Y^{4^{r-1}} a(X)) ~ \varepsilon_r^{-1} \ldots
\varepsilon_2^{-1} = \prod_{t=1}^k E_{p_t q_t}(Y d_t(X, Y))$, where $d_t (X, Y) \in I[X, Y]$,
and either $p_t =1$ or $q_t =1$. We now establish the result when $\varepsilon$ is product of 
$r$ many elementary generators. We have
\begin{eqnarray*}
&& \varepsilon ~ E_{ij}(Y^{4^r} a(X)) ~ \varepsilon^{-1}  \\
& = & \varepsilon_1 \varepsilon_2 \ldots \varepsilon_r ~ E_{ij}(Y^{4^r} a(X)) 
~ \varepsilon_r^{-1} \ldots \varepsilon_2^{-1} \varepsilon_1^{-1} \\
& = & \varepsilon_1 ~ \big( \prod_{t=1}^k E_{p_t q_t}(Y^4 d_t(X, Y)) \big) ~ \varepsilon_1^{-1}  \\
& = & \prod_{t=1}^k \varepsilon_1 ~ E_{p_t q_t}(Y^4 d_t(X, Y)) ~ \varepsilon_1^{-1} \\
& = & \prod_{t=1}^s E_{i_t j_t}(Y b_t(X, Y)). 
\end{eqnarray*}

To get the last equality one needs to repeat the computation which was
done for a single elementary generator. Note that here $b_t(X,Y) \in I[X,Y]$, and  
either $i_t=1$ or $j_t=1$.
\end{proof}

We now establish a dilation principle for the relative elementary linear transvection
group $\ETrans(IN)$. Note that similar dilation
principles for $\ETrans(N)$ and  $\ETrans(N, IN)$ were proved in \cite{bbr}, Proposition
3.1 and in \cite{cr2}, Lemma 4.6 respectively.

\begin{lem} \label{linear,dil,true-rel} Let $I$ be an ideal of $R$ and $M$ 
be a finitely generated module of $R$. Suppose that $a$ is a
non-nilpotent element of $R$ such that $M_a$ is a free $R_a$-module of
rank $n \ge 2$. Let $N=(R \oplus M)$. Let $\alpha(X) \in \Aut(N[X])$
with $\alpha(0)=Id.$, and $\alpha(X)_a \in
\ETrans(IN_a[X])$. Then there exists $\alpha^*(X) \in
\ETrans(IN[X])$, with $\alpha^*(0) = Id.$, such that
$\alpha^*(X)$ localises to $\alpha(bX)$, for $b \in (a^d)$, $d \gg 0$.
\end{lem}

\begin{proof} 
Using Lemma \ref{linear,free,true-rel} we get that 
$\ETrans(IN_a[X]) = \E_{n+1}^1(I_a[X])$. Therefore, one can write $\alpha(X)_a = \prod_t
\gamma_t E_{i_t j_t}(X f_t(X)) \gamma_t^{-1}$, where $\gamma_t \in \E_n^1(I_a)$, 
$f_t(X) \in I_a[X]$, and either $i_t =1$ or $j_t =1$. Using Lemma \ref{commE(I)} we write
\begin{eqnarray*}
\alpha(Y^{4^r}X)_a &=& \prod_k E_{i_k j_k}(Y g_k(X,Y)) \\
&=& \prod_k E_{i_k j_k}(Y h_k(X,Y)/ a^m),
\end{eqnarray*}
where $h_k(X,Y) \in I[X,Y]$, and either $i_k=1$ or $j_k=1$. Note that here $m$ is the maximum 
power of $a$ appearing in the denominators of $g_k(X, Y)$.

Since $N_a$ is a free $R_a$ module we have $N_a \cong R_a^{n+1} \cong
N_a^*$. Let $p_1, \ldots, p_{n+1}$ be the standard
basis of $N_a$, $p_1^*, \ldots, p_{n+1}^*$ be the standard basis of
$N_a^*$, and $e_1, \ldots, e_{n+1}$ be the standard basis of
$R_a^{n+1}$. Let $s \ge 0$ be an integer such that $\widetilde{p_i} =
a^s p_i \in N$, and $\widetilde{p_i^*} = a^s p_i^* \in N^*$, for all
$i$.  Note that
\begin{eqnarray*}
\alpha(Y^{4^r} X)_a &=& \prod_k \big( Id. + (Y h_k(X,Y)/ a^m) ~ e_{i_k} \cdot ~ e_{j_k}^t \big) \\
&=& \prod_k \big( Id. + (Y h_k(X,Y)/ a^m) ~ p_{i_k} \cdot ~ p_{j_k}^* \big)
\end{eqnarray*}

Let us set $d'=2s+m$. Now,
\begin{eqnarray*}
\alpha((a^{d'}Y)^{4^r} X)_a &=& \prod_k \big( Id. + a^{2s}Y
h_k(X,a^{d'}Y)) ~ p_{i_k} \cdot ~ p_{j_k}^* \big) \\ &=& \prod_k \big( Id. +
Y h_k(X,a^{d'}Y)) ~ a^s p_{i_k} \cdot ~ a^s p_{j_k}^* \big).
\end{eqnarray*}

Substituting $Y=1$ we get that $\alpha(a^d X)_a = \prod_k \big( Id. + 
h'_k(X) a^s p_{i_k} \cdot a^s p_{j_k}^* \big)$. Let us set $\alpha^*(X) =
\prod_k \big( Id. +  h'_k(X) \widetilde{p_{i_k}}
\cdot \widetilde{p_{j_k}^*} \big)$. Note that $\alpha^*(X)$ belongs to
$\ETrans(IN[X])$. It is clear from the construction 
that $\alpha^*(0)=Id.$ and $\alpha^*(X)$ localises to $\alpha(bX)$, 
for $b \in (a^d)$, $d \gg 0$.
\end{proof}

Next we establish a Local-Global principle for $\ETrans(IN)$.

\begin{lem} \label{linear,LG,true-rel} Let $I$ be an ideal of $R$ and let
  $M$ be a finitely generated projective $R$-module of rank $n \ge 2$.
  Let $N=(R \oplus M)$. Let $\alpha(X) \in \Aut(N[X])$, with $\alpha(0)
  =Id$. If for each maximal ideal $\gm$ of $R$, $\alpha(X)_\gm \in
  \ETrans(IN_\gm[X])$, then $\alpha(X) \in
  \ETrans(IN[X])$.
\end{lem}

\begin{proof} 
One can suitably choose an element $a_\gm$ from $R \setminus
\gm$ such that $\alpha(X)_{a_\gm}$ belongs to
$\ETrans(IN_{a_\gm}[X])$. Let us set $\gamma(X,Y) =
\alpha(X+Y) \alpha(Y)^{-1}$. Note that $\gamma(X,Y)_{a_\gm}$ belongs
to $\ETrans(IN_{a_\gm}[X,Y])$, and $\gamma(0,Y) =
Id$. From Lemma \ref{linear,dil,true-rel} it follows that $\gamma(b_\gm
X,Y) \in \ETrans(IN[X,Y])$, for $b_\gm \in (a_\gm^d)$, where $d
\gg 0$. Note that the ideal generated by $a_\gm^d$'s is the whole ring
$R$. Therefore, $c_1 a_{\gm_1}^d+ \ldots + c_k a_{\gm_k}^d = 1 $, for
some $c_i \in R$. Let $b_{m_i}= c_i a_{m_i}^d \in (a_{m_i}^d)$. It is
easy to see that $\alpha(X)=\prod_{i=1}^{k-1}\gamma(b_{m_i}X,T_i)
\gamma(b_{m_k}X,0)$, where $T_i = b_{m_{i+1}}X+ \cdots +
b_{m_k}X$. Each term in the right hand side of this expression belongs
to $\ETrans(IN[X])$ and hence $\alpha(X) \in \ETrans(IN[X])$.  
\end{proof}

Before we establish Tits's theorem for the linear transvection group, 
we state a lemma to show that every linear transvection is
homotopic to identity. Note that for any abelian group $G$, we
understand by $G[X]$ the abelian group of all polynomials in $X$ with
coefficients in $G$.

\begin{lem} \label{lin-trans-hom-to-identity}
Let $M$ be a finitely generated $R$-module and $\alpha \in
\Trans(M)$. Then there exists $\beta(X) \in \Trans(M[X])$ such that
$\beta(1)=\alpha$ and $\beta(0)=Id.$
\end{lem}

\begin{proof} 
See proof of Lemma 4.9 in \cite{cr2}.
\end{proof}

\begin{thm} \label{Tits-linTrans} Let $I$ be an ideal of $R$, $M$ be
  a finitely generated projective $R$-module of rank at least 2, and
  $N=(R \oplus M)$. Then $\ETrans(N,I^4N) \subseteq \ETrans(IN)$.
\end{thm}

\begin{proof} Let us consider an element $\alpha \in \ETrans(N,I^4N)$. Note 
that $\alpha$ is product of elements of the form $g_1 E(z) g_1^{-1}$
and $g_2 E^*(\tau) g_2^{-1}$, where $g_1, g_2 \in \ETrans(N)$,
$z \in I^4M$, and $\tau \in I^4M^*$. Let us set $\alpha(X)$ to be 
product of elements of the form $g_1(X) E(z X^4) g_1(X)^{-1}$
and $g_2(X) E^*(\tau X^4) g_2(X)^{-1}$, where $g_1(X)$ and  $g_2(X) \in \ETrans(N[X])$,
$z X^4 \in I^4M[X^4]$, and $\tau X^4 \in I^4M^*[X^4]$. Localizing
at a maximal ideal $\gm$ of $R$, we get $\alpha(X)_\gm$ is product
of elements of the form $g_1(X)_\gm ~ E(z X^4)_\gm ~ g_1(X)_\gm ^{-1}$
and $g_2(X)_\gm ~ E^*(\tau X^4)_\gm ~ g_2(X)_\gm ^{-1}$. These 
elements belong to $\ETrans(N_\gm[X],I^4N_\gm[X^4])$ which is equal to the group $\E_{n+1}(R_\gm[X], (I_\gm[X])^4)$ by Lemma \ref{linear,free,rel}. Using Theorem \ref{tits-lin} we get  
that $\alpha(X)_\gm \in  \E_{n+1}((I_\gm[X])^2) \subseteq \E_{n+1}^1(I_\gm[X])$
(see Remark \ref{inclusion-linear}). Note that $\E_{n+1}^1(I_\gm[X])= \ETrans(IN_\gm[X])$ by Lemma \ref{linear,free,true-rel}.  Using Lemma \ref{linear,LG,true-rel}, we get that $\alpha(X) \in
\ETrans(IN[X])$. Substituting $X=1$ we get that $\alpha \in \ETrans(IN)$.
\end{proof}

\begin{rmk} \label{rmk-linTrans}
{\rm We recall Tits's result for the elementary linear group which says that $\E_n(R, I^2) 
\subseteq \E_n(I)$ (see Theorem \ref{tits-lin}). However, our theorem as above involves fourth power 
of an ideal. In view of Tits's result it will be interesting to know if the 
above theorem holds for square of an ideal. }
\end{rmk}

\section{\large{The case of symplectic transvections groups}}

In this section we prove an analogue of Tits's result in the case of symplectic transvections groups.
As in the previous section we begin by recalling some associated basic notions in this context. 

\begin{de} {\bf Alternating Matrix:} {\rm  A matrix in ${\rm M}_n (R)$ is said
  to be {\it alternating} if it has the form $\nu - \nu^t$, where $\nu
  \in {\rm M}_n(R)$. It follows immediately that the diagonal elements of an alternating matrix are
  zeros.}
\end{de}

\begin{lem} \label{local,rel} Let $(R,\gm)$ be a local ring and $I$ be an
  ideal of $R$. Let $\varphi$ be an alternating matrix of Pfaffian 1
  over $R$, and $\varphi \equiv \psi_n ~({\rm mod}~I)$. Then $\varphi$
  is of the form $$(1 \perp \varepsilon)^t \psi_n (1 \perp
  \varepsilon),$$ for some $\varepsilon \in {\rm E}_{2n-1}(R,I)$.
\end{lem}

\begin{proof}
Follows using induction on $n$ and the fact that for any $v \in
\Um_n(R,I)$, $v=e_1 \beta$, for some $\beta \in \E_n(R,I)$ over the local ring $(R, \gm)$. For details one 
can see proof of Lemma 5.2 in \cite{cr2}.
\end{proof}

\begin{de} 
{\rm 
A {\it symplectic $R$-module} is a pair $(P,\langle , \rangle)$,
where $P$ is a finitely generated projective $R$-module of even rank
and $\langle , \rangle: P \times P \lra R$ is a non-degenerate (i.e, $P
\cong P^*$ by $x \lra \langle x , -\rangle$) alternating bilinear form.
}
\end{de}

\begin{de}
{\rm
Let $(P_1,\langle , \rangle_1)$ and $(P_2,\langle , \rangle_2)$ be
two  symplectic $R$-modules. Their {\it orthogonal sum} is the pair
$(P,\langle , \rangle)$, where $P=P_1 \oplus P_2$ and the inner product
is defined by $\langle (p_1,p_2),(q_1,q_2)\rangle = \langle
p_1,q_1\rangle_1 + \langle p_2,q_2\rangle_2$.

There is a non-degenerate alternating bilinear form $\langle , \rangle$ on the
$R$-module $\mathbb{H}(R)= R \oplus R^*$, namely $\langle (a_1,f_1),
(a_2,f_2) \rangle = f_2(a_1) - f_1(a_2)$.  }
\end{de}

\begin{de}
{\rm An {\it isometry} of a symplectic module $(P,\langle , \rangle)$ is
  an automorphism of $P$ which fixes the bilinear form. The group of
  isometries of $(P, \langle , \rangle)$ is denoted by
  $\Sp(P)$.  }
\end{de}

\begin{de} 
{\rm 
In \cite{B2} Bass has defined a {\it symplectic transvection}
of a symplectic module $(P,\langle , \rangle)$ to be an automorphism of the form
\begin{eqnarray*}
\sigma(p) &=& p + \langle u , p \rangle v + \langle v , p \rangle u + \alpha
\langle u , p \rangle u,
\end{eqnarray*}
where $\alpha \in R$, $u,v \in P$ are fixed elements with $\langle
u,v \rangle=0$, and either $u$ or $v$ is unimodular. It is easy to
check that $\langle \sigma(p), \sigma(q) \rangle = \langle p, q
\rangle$ and $\sigma$ has an inverse $\tau(p) = p - \langle u, p
\rangle v - \langle v, p \rangle u - \alpha \langle u, p \rangle u$.

The subgroup of $\Sp(P)$ generated by the symplectic
transvections is denoted by $\Trans_{\Sp}(P, \langle , \rangle)$ (see
\cite{Sw}, Page 35).  }
\end{de}

{\it Now onwards $Q$ will denote $(R^2 \oplus P)$ with the induced form
  on $(\mathbb{H}(R)~\oplus~P)$, and $Q[X]$ will denote $(R[X]^2 \oplus
  P[X])$ with the induced form on $(\mathbb{H}(R[X])~\oplus~P[X])$.}

\begin{de} {\rm Symplectic transvections of $Q=(R^2 \oplus P)$ of 
the form
    \begin{eqnarray*}
      (a, b, p) & \mapsto & (a, b - \langle p, q \rangle - \alpha a, p-aq),
     \end{eqnarray*} 
or of the form 
    \begin{eqnarray*}
      (a, b, p) & \mapsto & (a + \langle p, q \rangle + \beta b, b,  p-bq), 
    \end{eqnarray*} 
where $\alpha, \beta \in R$ and $q \in P$, 
are called {\it elementary symplectic transvections}. Let us denote the first 
isometry by $\rho(q, \alpha)$ and the second one by $\mu(q, \beta)$. It can 
be verified that the elementary symplectic transvections are symplectic 
transvections on $Q$. Indeed, consider $(u, v) =((0,1,0),(0,0,q))$  to get 
$\rho(q, \alpha)$ and consider $(u, v)= ((-1,0,0),(0,0,q))$ to get $\mu(q, \beta)$.

The subgroup of $\Trans_{\Sp}(Q, \langle , \rangle)$ generated by the
elementary symplectic transvections is denoted by $\ETrans_{\Sp}(Q,
\langle , \rangle)$.}
\end{de}

\begin{de} {\rm Let $I$ be an ideal of $R$. Elementary symplectic
transactions of $Q$ of the form $\rho(q, \alpha), \mu(q, \beta)$, where
$q \in IP$ and $\alpha, \beta \in I$ are called {\it relative elementary 
symplectic transvections} with respect to an ideal $I$.

The subgroup of $\ETrans_{\Sp}(Q, \langle , \rangle)$ generated by
relative elementary symplectic transvections is denoted by
$\ETrans_{\Sp}(IQ,\langle , \rangle )$. The normal closure of
$\ETrans_{\Sp}(IQ,\langle , \rangle )$ in
$\ETrans_{\Sp}(Q, \langle , \rangle)$ is denoted by
$\ETrans_{\Sp}(Q,IQ, \langle , \rangle)$ }.
\end{de}

\begin{rmk} \label{free} {\rm Let $P=\oplus_{i=1}^{2n} Re_i$ be a free $R$-module. 
The non-degenerate alternating bilinear form $\langle,\rangle$ on $P$
corresponds to an alternating matrix $\varphi$ with Praffian 1 with
respect to the basis $\{e_1, e_2, \ldots, e_{2n} \}$ of $P$ and we write
$\langle p, q \rangle = p \varphi q^t$.

  In this case the symplectic transvection $\sigma(p) = p + \langle u,
  p \rangle v + \langle v, p \rangle u + \alpha \langle u, p \rangle
  u$ corresponds to the matrix $(I_{2n} + v u^t \varphi  + u v^t \varphi)
  (I_{2n} + \alpha u u^t \varphi )$ and the group generated by them
  is denoted by $\Trans_{\Sp}(P, \langle , \rangle_{\varphi})$.

Also in this case $\ETrans_{\Sp}(Q, \langle , \rangle_{\psi_1 \perp \varphi})$ will
be generated by the matrices of the form $\rho_{\varphi}(q, \alpha) =
\Big( \begin{smallmatrix} 1 & 0 & 0 \\ -\alpha  & 1 & q^t \varphi \\ -q &
  0 & I_{2n} \end{smallmatrix} \Big)$, and $\mu_{\varphi}(q, \beta) =
\Big( \begin{smallmatrix} 1 & \beta & -q^t \varphi \\ 0 & 1 & 0 \\ 0 &
-q & I_{2n} \end{smallmatrix} \Big)$.

Note that for $q=(q_1, \ldots, q_{2n}) \in R^{2n}$, and for the
standard alternating matrix $\psi_n$, we have
\begin{eqnarray} \label{relatn5}
\rho_{\psi_n}(q, \alpha) &=& se_{21}(-\alpha + q_1q_2 + \cdots + q_{2n-1}q_{2n}) \prod_{i=3}^{2n+2}
se_{i1}(-q_{i-2}), \\ 
\label{relatn6}
\mu_{\psi_n}(q, \beta) &=& se_{12}(\beta + q_1q_2 + \cdots + q_{2n-1}q_{2n})
\prod_{i=3}^{2n+2} se_{1i}((-1)^{i+1} q_{\sigma(i-2)}).  
\end{eqnarray}

We shall implicitly use these assumptions and notations in this section.}
\end{rmk}

\begin{lem} \label{symp,free,true-rel} 
Let $R$ be a commutative ring and $I$ be an ideal of $R$. Let $(P, \langle, \rangle_{\varphi})$
be a symplectic $R$-module with $P$ be free of rank $2n$, $n \ge 1$ and $Q=(R^2 \oplus P)$ 
with the induced form on $(\mathbb{H}(R)~\oplus~P)$. If $\varphi = \psi_n$ is
the standard alternating matrix, then $\ETrans_{\Sp}(IQ, \langle ,
\rangle_{\psi_{n+1}}) = \ESp_{2n+2}^1(I)$.
\end{lem}

\begin{proof}
In this case $\ETrans_{\Sp}(IQ, \langle ,\rangle_{\psi_{n+1}})$ is 
generated by the elements of the form $\rho_{\psi_n}(q, \alpha)$ and 
$\mu_{\psi_n}(q, \beta)$, where $q \in I^{2n} (\subseteq R^{2n})$, and $\alpha,
\beta \in I$. By Remark \ref{free} it follows that $\ETrans_{\Sp}(IQ, \langle ,\rangle_{\psi_{n+1}}) 
\subseteq \ESp_{2n+2}^1(I)$. When $Q$ is free, generators of $\ESp_{2n+2}^1(I)$ belong 
to $\ETrans_{\Sp}(IQ, \langle , \rangle_{\psi_{n+1}})$, 
and hence we get that $\ESp_{2n+2}^1(I) \subseteq \ETrans_{\Sp}(IQ, \langle , 
\rangle_{\psi_{n+1}})$.
\end{proof}

\begin{lem} \label{symp,free,rel} Let $R$ be a commutative ring with
  $R=2R$, and let $I$ be an ideal of $R$. Let  $(P, \langle, \rangle_{\varphi})$ be a symplectic $R$-module
  with $P$ be free of rank $2n$, $n \ge 1$ and $Q=(R^2 \oplus P)$ 
with the induced form on $(\mathbb{H}(R)~\oplus~P)$. If $\varphi = \psi_n$ is the standard
  alternating matrix, then $\ETrans_{\Sp}(Q,IQ,
  \langle , \rangle_{\psi_{n+1}}) = \ESp_{2n+2}(R,I)$.
\end{lem}

\begin{proof}
See proof of Lemma 5.14 in \cite{cr2}.
\end{proof}

\begin{lem} \label{phi=phi*,true-relative}
Let $P$ be a free $R$-module of rank $2n$. Let $(P,\langle ,
\rangle_{\varphi})$ and $(P,\langle , \rangle_{\varphi^*})$ be two
symplectic $R$-modules with $\varphi= (1 \perp \varepsilon)^t ~
\varphi^* ~ (1 \perp \varepsilon)$, for some $\varepsilon \in
\E_{2n-1}(R)$, and $Q=(R^2 \oplus P)$ with the induced form
  on $(\mathbb{H}(R)~\oplus~P)$. Then
\begin{eqnarray*}  
\ETrans_{\Sp}(IQ, \langle , \rangle_{\psi_1 \perp \varphi}) &=& (I_3 \perp
\varepsilon)^{-1}~ \ETrans_{\Sp}(IQ,\langle , \rangle_{\psi_1 \perp \varphi^*}) ~ (I_3
\perp \varepsilon), \\
\ETrans_{\Sp}(Q, IQ, \langle , \rangle_{\psi_1 \perp \varphi}) &=& (I_3 \perp
\varepsilon)^{-1}~ \ETrans_{\Sp}(Q, IQ,\langle , \rangle_{\psi_1 \perp \varphi^*}) ~ (I_3
\perp \varepsilon).
\end{eqnarray*}
\end{lem}

\begin{proof}
For elementary symplectic transvections we have 
\begin{eqnarray*}
(I_2 \perp (1 \perp \varepsilon))^{-1} \rho_{\varphi^*}(q, \alpha) (I_2
  \perp (1 \perp \varepsilon) = \rho_{\varphi} ((1 \perp
  \varepsilon)^{-1} q, \alpha),\\ (I_2 \perp (1 \perp
  \varepsilon))^{-1} \mu_{\varphi^*}(q, \beta) (I_2 \perp (1 \perp
  \varepsilon) = \mu_{\varphi} ((1 \perp
  \varepsilon)^{-1} q, \beta).
\end{eqnarray*}
Hence the equalities follow.
\end{proof}

\begin{lem} \label{commESp(I)} Let $R$ be a ring with $R=2R$, $I$ be an ideal of $R$,
and $n \ge 2$. Let $\varepsilon =
  \varepsilon_1 \ldots \varepsilon_r$,
  where each $\varepsilon_k$ is an elementary generator of $\ESp_{2n}^1(I)$.  If
  $se_{ij}(a(X)) \in \ESp_{2n}^1(I[X])$, then
\begin{eqnarray*}
\varepsilon ~ se_{ij}(Y^{4^r} a(X)) ~ \varepsilon^{-1} 
&=& \prod_{t=1}^s se_{i_t j_t}(Y b_t(X,Y)),
\end{eqnarray*}
where $b_t (X,Y) \in I[X, Y]$, and either $i_t =1$ or $j_t =1$.
\end{lem}

\begin{proof}
We prove this result using induction on $r$, where $\varepsilon$ is
product of $r$ many elementary generators. First we state the following identities
\begin{align}\label{commprop}
[gh,k] &~ = ~ \big({}^g[h,k]\big)[g,k],\\
[g,hk] &~ = ~ [g,h]\big({}^h[g,k]\big),\\
{}^g[h,k] &~ = ~ [{}^gh,{}^gk],
\end{align}
where $^gh$ denotes $ghg^{-1}$ and $[g,h]=ghg^{-1}h^{-1}$. These
identities will be used throughout the proof. Also, we need the following 
relations which hold for all integers $i,j$ with $i \neq j, \sigma(j)$, and for 
all $a, b \in R$.
\begin{align}\label{commprop2}
[se_{i \sigma(i)}(a), ~ se_{\sigma(i) j}(b)] &~ = ~ se_{ij}(ab) ~ se_{\sigma(j) j}((-1)^{i+j}ab^2), \\ 
[se_{ik}(a), ~ se_{kj}(b)] & ~ = ~ se_{ij}(ab), {\rm if} ~ k \ne \sigma(i), \sigma(j), \\ 
[se_{ik}(a), ~ se_{k \sigma(i)}(b)] & ~ = ~ se_{i \sigma(i)}(2ab), {\rm if} ~ k \ne i, \sigma(i).
\end{align}
The following relation holds for all $i, j$ with $1 \le i\ne j \le
2n$, and for all $a, b \in R$.
\begin{align}
[se_{ij}(a), ~ se_{kl}(b)] & ~=~ Id., {\rm if} ~ i \ne l, \sigma(k) ~ {\rm
  and} ~ j \ne k, \sigma(l).
\end{align}
Let us assume $r=1$, i.e, $\varepsilon
= se_{pq}(c)$, where $c \in I$. Also, for the elementary generator $se_{ij}(Y^{4^r} a(X))$
we assume $i=1$. First we consider the case when $j = \sigma(i)$, i.e, $j =2$.

{\it Case $($1$)$:} Let $(p,q)=(1,2)$. In this case we get that 
\begin{eqnarray*}
se_{12}(c) ~ se_{12}(Y^4 a(X)) ~ se_{12}(-c) &=& se_{12}(Y^4 a(X)).
\end{eqnarray*}

{\it Case $($2$)$:} Let $(p,q) = (1,k)$ and $k \ne 2$. In this case we get that 
\begin{eqnarray*}
se_{1k}(c) ~ se_{12}(Y^4 a(X)) ~ se_{12}(-c) &=& se_{12}(Y^4 a(X)).
\end{eqnarray*}

{\it Case $($3$)$:} Let $(p,q)=(2, 1)$. In this case we get that 
\begin{eqnarray*}
&& se_{21}(c) ~ se_{12}(Y^4 a(X)) ~ se_{21}(-c) \\ 
&=& ^{se_{21}(c)} [se_{1k}(Y^2), ~ se_{k 2}(Y^2 a'(X))], ~{\rm where}~k \ne 1,2 ~ {\rm and} a'(X) = a(X)/2\\ 
&=& [^{se_{21}(c)} se_{1k}(Y^2), ~ se_{2 \sigma(k)}((-1)^{k+1} Y^2 c a'(X)) \\
&& se_{k \sigma(k)}((-1)^{k+1} Y^4 c a'(X)^2) ~ se_{k 2}(Y^2 a'(X))] 
\end{eqnarray*}
\begin{eqnarray*}
&=& [\alpha, ~ se_{k 1}(-Y^2 c a'(X)) ~ se_{k \sigma(k)}((-1)^{k+1} Y^4 c a'(X)^2) 
 ~ se_{k 2}(Y^2 a'(X))] \\ 
 &=& ^{\alpha} se_{k 1}(-Y^2 c a'(X)) ~ ^{\alpha} se_{k \sigma(k)}((-1)^{k+1} Y^4 c a'(X)^2) 
 ~ ^{\alpha} se_{k 2}(Y^2 a'(X)) \\
 && se_{k 2}(-Y^2 a'(X)) ~ se_{k \sigma(k)}((-1)^k Y^4 c a'(X)^2) ~ se_{k 1}(Y^2 c a'(X)) \\
 &=& ^{\alpha} se_{k 1}(-Y^2 c a'(X)) ~ ^{\alpha} se_{k \sigma(k)}((-1)^{k+1} Y^4 c a'(X)^2) 
 ~ ^{\alpha} se_{k 2}(Y^2 a'(X)) \\
 && se_{k 2}(-Y^2 a'(X)) ~ [se_{k 1}(Y^2 c), ~ se_{1 \sigma(k)}((-1)^k Y^2 a'(X))] ~ se_{k 1}(Y^2 c a'(X)).
\end{eqnarray*}

The element $\alpha = ^{se_{21}(c)} se_{1k}(Y^2)$ in the above computation can be expressed in the following three different ways, namely,
\begin{eqnarray*}
\alpha &=& se_{1 k}(Y^2) ~ se_{2 k}(Y^2 c) ~ se_{\sigma(k) k}(Y^4 c) \\ 
&=& se_{2 k}(Y^2 c) ~ se_{\sigma(k) k}(Y^4 c) ~ se_{1k}(Y^2) \\ 
&=& se_{\sigma(k) k}(Y^4 c) ~ se_{2 k}(Y^2 c) ~ se_{1k}(Y^2).
\end{eqnarray*} 
We will use the above expressions of $\alpha$ as and when we find it convenient. Now 
\begin{eqnarray*}
&& ^{\alpha} se_{k 1}(-Y^2 c a'(X)) \\
&=& ^{se_{1 k}(Y^2) ~ se_{2 k}(Y^2 c) ~ se_{\sigma(k) k}(Y^4 c)} se_{k 1}(-Y^2 c a'(X))  \\ 
&=& ^{se_{1 k}(Y^2) ~ se_{2 k}(Y^2 c)} \big( se_{\sigma(k) 1}(-Y^6c^2 a'(X)) ~ se_{21}((-1)^k Y^8 c^3 a'(X)^2) \\
&&  se_{k 1}(-Y^2 c a'(X)) \big) \\ 
&=& ^{se_{1k}(Y^2)} \big( se_{\sigma(k) 1}(-Y^6c^2 a'(X)) ~  se_{21}((-1)^k Y^8 c^3 a'(X)^2 - 2 Y^4 c^2 a'(X)) \big) \\ 
&=& ^{se_{1k}(Y^2)} \big( se_{\sigma(k) 1}(-Y^6c^2 a'(X)) ~  se_{21}(Y^4 \circledast) \big) ~{\rm where} ~ \circledast \in I^2[X] \\ 
&=& se_{\sigma(k) k}(Y^8 c^2 a'(X)) ~ se_{\sigma(k) 1}(-Y^6c^2 a'(X)) ~ se_{2k}(Y^6 \circledast) ~ se_{\sigma(k) k}((-1)^{k+1}Y^8 \circledast)\\
&=& [se_{\sigma(k) 1}(Y^4 c^2), ~ se_{1 k}(Y^4 a'(X)] ~ se_{\sigma(k) 1}(-Y^6c^2 a'(X)) ~ se_{2k}(Y^6 \circledast) \\
&&  [se_{\sigma(k) 1}((-1)^{k+1} Y^4 *), ~ se_{1 k}(Y^4 *)], \\
&& ^{\alpha} se_{k 2}(Y^2 a'(X)) \\
&=& ^{se_{2 k}(Y^2 c) ~ se_{\sigma(k) k}(Y^4 c) ~ se_{1k}(Y^2)} se_{k 2}(Y^2 a'(X)) \\ 
&=& ^{se_{2 k}(Y^2 c) ~ se_{\sigma(k) k}(Y^4 c)} \big( se_{12}(2 Y^4 a'(X)) ~  se_{k 2}(Y^2 a'(X)) \big) \\ 
&=& ^{se_{2 k}(Y^2 c)} \big( se_{12}(2 Y^4 a'(X)) ~ se_{\sigma(k) 2}(Y^6 c a'(X)) ~ se_{12}((-1)^{k+1} y^8 c a'(X)^2) \\  &&se_{k 2}(Y^2 a'(X)) \big), \\
&& ^{\alpha} se_{k \sigma(k)}((-1)^{k+1} Y^4 c a'(X)^2) \\
&=& ^{se_{\sigma(k) k}(Y^4 c) ~ se_{2 k}(Y^2 c) ~ se_{1k}(Y^2)} ~ se_{k \sigma(k)}((-1)^{k+1} Y^4 c a'(X)^2) \\ 
&=& ^{se_{\sigma(k) k}(Y^4 c) ~ se_{2 k}(Y^2 c)} \big( se_{k2}((-1)^k Y^6 c a'(X)^2) ~ se_{12}(Y^8 c a'(X)^2) \\
&& se_{k \sigma(k)}((-1)^{k+1} Y^4 c a'(X)^2) \big) \\ 
&=& XYZ ~{\rm(say)},
\end{eqnarray*} 
where
\begin{eqnarray*}
X &=& ^{se_{\sigma(k) k}(Y^4 c) ~ se_{2 k}(Y^2 c)}  se_{k2}((-1)^k Y^6 c a'(X)^2) \\ 
&=& ^{se_{2 k}(Y^2 c) ~ se_{\sigma(k) k}(Y^4 c)}  se_{k2}((-1)^k Y^6 c a'(X)^2) \\ 
&=& ^{se_{2 k}(Y^2 c)} \big(  se_{\sigma(k) 2}((-1)^k Y^8 c^2 a'(X)^2)  ~ se_{12} ((-1)^{k+1}Y^{16} c^3 a'(X)^4) \\
&&  se_{k2}((-1)^k Y^6 c a'(X)^2)\big), \\
Y &=& ^{se_{\sigma(k) k}(Y^4 c) ~ se_{2 k}(Y^2 c)} se_{12}(Y^8 c a'(X)^2) \\ 
&=&  ^{se_{\sigma(k) k}(Y^4 c)} \big( se_{1k}(Y^8 c^2 a'(X)^2) ~ se_{\sigma(k) k}((-1)^k Y^{12} c^3 a'(X)^2) \\
&&  se_{12}(Y^8 c a'(X)^2) \big) \\
&=& se_{1k}(Y^8 c^2 a'(X)^2) ~ se_{\sigma(k) k}((-1)^k Y^{12} c^3 a'(X)^2) ~ se_{12}(Y^8 c a'(X)^2), \\
Z &=& ^{se_{\sigma(k) k}(Y^4 c) ~ se_{2 k}(Y^2 c)} se_{k \sigma(k)}((-1)^{k+1} Y^4 c a'(X)^2) \\ 
&=& ^{se_{\sigma(k) k}(Y^4 c)} \big( se_{k 1}(-Y^6 c^2 a'(X)^2) ~ se_{21}((-1)^k Y^8 c^3 a'(X)^2) \\ 
&& se_{k \sigma(k)}((-1)^{k+1} Y^4 c a'(X)^2) \big) \\
&=& se_{\sigma(k) 1}(-Y^{10}c^3 a'(X)^2)  ~ se_{21}((-1)^kY^{16} c^5 a'(X)^4) ~ se_{21}((-1)^k Y^8 c^3 a'(X)^2) \\
&& ^{se_{\sigma(k) k}(Y^4 c)} se_{k \sigma(k)}((-1)^{k+1} Y^4 c a'(X)^2) \\
&=& se_{\sigma(k) 1}(-Y^{10}c^3 a'(X)^2)  ~ se_{21}((-1)^kY^{16} c^5 a'(X)^4) ~ se_{21}((-1)^k Y^8 c^3 a'(X)^2) \\
&& ^{se_{\sigma(k) k}(Y^4 c)} [se_{k 1}(Y^2 a'(X)^2), ~ se_{1 \sigma(k)}((-1)^{k+1} Y^2 c)] \\
&=& se_{\sigma(k) 1}(-Y^{10}c^3 a'(X)^2)  ~ se_{21}((-1)^kY^{16} c^5 a'(X)^4) ~ se_{21}((-1)^k Y^8 c^3 a'(X)^2) \\
&& [se_{\sigma(k) 1}(Y^6 c a'(X)^2) ~ se_{21}((-1)^k Y^8 c a'(X)^4 ~ se_{k 1}(Y^2 a'(X)^2), \\
&& se_{\sigma(k) 2}(Y^6 c^2) ~ se_{12}((-1)^{k+1} Y^8 c^3) ~ se_{1 \sigma(k)}((-1)^{k+1} Y^2 c)].
\end{eqnarray*}

{\it Case $($4$)$:} Let $(p,q)=(k,1)$, where $k \ne 2$. In this case we get that 
\begin{eqnarray*}
&& se_{k1}(c) ~ se_{12}(Y^4 a(X)) ~ se_{k1}(-c) \\
&=& [se_{12}(Y^4 a(X)), ~ se_{2 \sigma(k)}((-1)^k c)]^{-1} se_{12}(Y^4 a(X))\\
&=& \{ se_{1 \sigma(k)}((-1)^k Y^4 c a(X)) ~ se_{k \sigma(k)}((-1)^k Y^4 c^2 a(X)) \}^{-1} ~ se_{12}(Y^4 a(X))\\
&=& \{ se_{1 \sigma(k)}((-1)^k Y^4 c a(X)) ~ [se_{k 1}((-1)^k Y^2 c^2), ~ se_{1 \sigma(k)}(Y^2 a(X))] \}^{-1}\\
&& se_{12}(Y^4 a(X)).
\end{eqnarray*}

Next we consider the elementary generator $se_{ij}(Y^{4^r} a(X))$ with $i=1$ and 
$j \ne \sigma(i)$.

\medskip

{\it Case $($5$)$:} Let $(p,q)=(1,2)$. In this case we get that
\begin{eqnarray*}
se_{12}(c) ~ se_{1j}(Y^4 a(X)) ~ se_{12}(-c) &=& se_{1j}(Y^4 a(X)).
\end{eqnarray*}

{\it Case $($6$)$:} Let $(p,q) = (1,k)$ and $k \ne 2, j, \sigma(j)$. In this case we get that 
\begin{eqnarray*}
se_{1k}(c) ~ se_{1j}(Y^4 a(X)) ~ se_{1k}(-c) &=& se_{1j}(Y^4 a(X)).
\end{eqnarray*}

{\it Case $($7$)$:} Let $(p,q) = (1,j)$. In this case we get that 
\begin{eqnarray*}
se_{1j}(c) ~ se_{1j}(Y^4 a(X)) ~ se_{1j}(-c) &=& se_{1j}(Y^4 a(X)).
\end{eqnarray*}

{\it Case $($8$)$:} Let $(p,q) = (1, \sigma(j))$ and $k \ne 2, j, \sigma(j)$. In this case we get that
\begin{eqnarray*}
se_{1 \sigma(j)}(c) ~ se_{1j}(Y^4 a(X)) ~ se_{1 \sigma(j)}(-c) &=& se_{12}((-1)^jY^4a(X)) ~ se_{1j}(Y^4 a(X)).
\end{eqnarray*}

{\it Case $($9$)$:} Let $(p,q) = (2, 1)$. In this case we get that
\begin{eqnarray*}
&& se_{21}(c) ~ se_{1j}(Y^4 a(X)) ~ se_{21}(-c) \\
&=& se_{2 j}(Y^4 c a(X)) ~ se_{\sigma(j) j}((-1)^j Y^8 c a(X)^2) ~ se_{1j}(Y^4 a(X)) \\
&=& se_{\sigma(j) 1}((-1)^{j+1} Y^4 c a(X)) ~ [se_{\sigma(j) 1}((-1)^4 c/2), ~ se_{1 j}(Y^4 a(X)^2)] \\
&& ~ se_{1j}(Y^4 a(X)).
\end{eqnarray*}

{\it Case $($10$)$:} Let $(p,q) = (k, 1)$ and $k \ne 2, j, \sigma(j)$. In this case we get that
\begin{eqnarray*}
&& se_{k 1}(c) ~ se_{1j}(Y^4 a(X)) ~ se_{k 1}(-c) \\
&=& se_{kj}(Y^4 c a(X)) ~ se_{1j}(Y^4 a(X)) \\
&=& [se_{k1}(Y^2 c), ~ se_{1j}(Y^2 a(X)] ~ se_{1j}(Y^4 a(X)).
\end{eqnarray*}

{\it Case $($11$)$:} Let $(p,q) = (j, 1)$. In this case we get that
\begin{eqnarray*}
&& se_{j 1}(c) ~ se_{1j}(Y^4 a(X)) ~ se_{j 1}(-c) \\
&=& se_{j 1}(c) ~ [se_{1k}(Y^2 a(X)), ~ se_{kj}(Y^2)] se_{j 1}(-c) \\
&=& [se_{jk}(Y^2 c a(X) ~ se_{1k}(Y^2 a(X)), ~ se_{k1}(-Y^2 c) ~ se_{kj}(Y^2)] \\
&=& se_{jk}(Y^2 c a(X) ~ se_{1k}(Y^2 a(X)) ~ se_{k1}(-Y^2 c) ~ se_{kj}(Y^2) \\ 
&& se_{1k}(-Y^2 a(X)) ~ se_{jk}(-Y^2 c a(X) ~ se_{kj}(-Y^2) ~ se_{k1}(Y^2 c) \\
&=& se_{jk}(Y^2 c a(X) ~ se_{1k}(Y^2 a(X)) ~ se_{k1}(-Y^2 c) ~ se_{1k}(-Y^2 a(X)) \\
&& se_{1k}(Y^2 a(X)) ~ se_{kj}(Y^2) ~ se_{1k}(-Y^2 a(X)) ~ se_{kj}(-Y^2) \\
&& se_{kj}(Y^2) ~ se_{jk}(-Y^2 c a(X) ~ se_{kj}(-Y^2) ~ se_{k1}(Y^2 c) \\
&=& [se_{j1}(Y c), ~ se_{1k}(Y a(X)] ~ se_{1k}(Y^2 a(X)) ~ se_{k1}(-Y^2 c) \\
&& se_{1k}(-Y^2 a(X)) ~ se_{1j}(Y^4 a(X)) ~ se_{kj}(Y^2) \\
&& [se_{j1}(Y c), ~ se_{1k}(-Y a(X)] ~ se_{kj}(-Y^2) ~ se_{k1}(Y^2 c) \\
&=& [se_{j1}(Y c), ~ se_{1k}(Y a(X)] ~ se_{1k}(Y^2 a(X)) ~ se_{k1}(-Y^2 c) \\
&& se_{1k}(-Y^2 a(X)) ~ se_{1j}(Y^4 a(X)) ~ [se_{k1}(Y^3 c) ~ se_{j1}(Y c), \\
&& se_{1j}Y^3 a(X) ~ se_{1k}(Y a(X)] ~ se_{k1}(Y^2 c).
\end{eqnarray*}

{\it Case $($12$)$:} Let $(p,q) = (\sigma(j), 1)$. In this case we get that
\begin{eqnarray*}
&& se_{\sigma(j) 1}(c) ~ se_{1j}(Y^4 a(X)) ~ se_{\sigma(j) 1}(-c) \\
&=& se_{\sigma(j) j}(2Y^4 c a(X)) ~ se_{1j}(Y^4 a(X)) \\
&=& [se_{\sigma(j) 1}(Y^2 c), ~ se_{1 j}(Y^2 a(X))] ~ se_{1j}(Y^4 a(X)).
\end{eqnarray*}

Hence the result is true when $i=1$ and $\varepsilon$ is an elementary
generator. Carrying out similar computations one can show that the result is true
when $j=1$ and $\varepsilon$ is an elementary generator.
Let us assume that the result is true when $\varepsilon$ is a product of $r-1$
many elementary generators, i.e, $\varepsilon_2 \ldots \varepsilon_r ~
se_{ij}(Y^{4^{r-1}} a(X)) ~ \varepsilon_r^{-1} \ldots
\varepsilon_2^{-1} = \prod_{t=1}^k se_{p_t q_t}(Y d_t(X, Y))$, where $d_t (X, Y) \in I[X, Y]$,
and either $p_t =1$ or $q_t =1$. We now establish the result when $\varepsilon$ is product 
of $r$ many
elementary generators. We have
\begin{eqnarray*}
&& \varepsilon ~ se_{ij}(Y^{4^r} a(X)) ~ \varepsilon^{-1}  \\
& = & \varepsilon_1 \varepsilon_2 \ldots \varepsilon_r ~ se_{ij}(Y^{4^r} a(X)) 
~ \varepsilon_r^{-1} \ldots \varepsilon_2^{-1} \varepsilon_1^{-1} \\
& = & \varepsilon_1 ~ \big( \prod_{t=1}^k se_{p_t q_t}(Y^4 d_t(X, Y)) \big) ~ \varepsilon_1^{-1}  \\
& = & \prod_{t=1}^k \varepsilon_1 ~ se_{p_t q_t}(Y^4 d_t(X, Y)) ~ \varepsilon_1^{-1} \\
& = & \prod_{t=1}^s se_{i_t j_t}(Y b_t(X, Y)). 
\end{eqnarray*}

To get the last equality one needs to repeat the computation which was
done for a single elementary generator. Note that here $b_t(X,Y) \in I[X,Y]$, and  
either $i_t=1$ or $j_t=1$.
\end{proof}

Next we establish a dilation principle for the relative elementary symplectic
transvection group $\ETrans_{\Sp}(IQ, \langle, \rangle)$. Note that similar dilation principles
for $\ETrans_{\Sp}(Q, \langle, \rangle)$ and  $\ETrans_{\Sp}(Q, IQ, \langle, \rangle)$  were 
proved in \cite{bbr}, Proposition 3.1 and in \cite{cr2}, Lemma 5.19 respectively.

\begin{lem} \label{symp,dil,true-rel} Let $R$ be a commutative ring
  with $R=2R$, and let $I$ be an ideal of $R$. Let $(P,\langle ,
  \rangle)$ be a symplectic $R$-module with $P$ finitely generated
  projective $R$-module of rank $2n$, $n \ge 1$ and $Q=(R^2 \oplus P)$ 
  with the induced form
  on $(\mathbb{H}(R)~\oplus~P)$. Suppose that $a$ is a
  non-nilpotent element of $R$ such that $P_a$ is a free $R_a$ module
  and the bilinear form $\langle, \rangle$ corresponds to the
  alternating matrix $\varphi$ (with respect to some basis). Assume that
  $\varphi = (1 \perp \varepsilon)^t ~ \psi_n ~ (1 \perp
  \varepsilon)$, for some $\varepsilon \in \E_{2n-1}(R_a)$. Let
  $\alpha(X) \in \Aut(Q[X])$, with $\alpha(0) = Id$, and $\alpha(X)_a 
  \in \ETrans_{\Sp}(IQ_a[X],\langle , \rangle_{\psi_1 \perp
    \varphi})$. Then, there exists $\alpha^*(X) \in 
    \ETrans_{\Sp}(IQ[X],\langle , \rangle)$, with $\alpha^*(0) =
  Id.$, such that $\alpha^*(X)$ localises to $\alpha(bX)$, for $b \in
  (a^d)$, $d \gg 0$.
\end{lem}

\begin{proof}
We have $P_a \cong R_a^{2n}$. Let $e_1, \ldots, e_{2n+2}$ be the standard basis of $Q_a$ with respect to which the bilinear form on $Q_a$ will correspond to $\psi_1 \perp \varphi$. It is given that $\alpha(X)_a \in \ETrans_{\Sp}(IQ_a[X], \langle, \rangle_{\psi_1 \perp \varphi})$. Using Lemma \ref{symp,free,true-rel}, and Lemma
\ref{phi=phi*,true-relative} we get that $\ETrans_{\Sp}(IQ_a[X], \langle, \rangle_{\psi_1 \perp \varphi}) = (I_3 \perp \varepsilon)^{-1} ~ \ESp_{2n+2}^1(I_a[X]) ~ (I_3 \perp \varepsilon)$. Therefore, $\alpha(X)_a = (I_3 \perp \varepsilon)^{-1} ~\beta(X)~ (I_3 \perp \varepsilon)$, for some $\beta(X) \in \ESp_{2n+2}^1(I_a[X])$, with $\beta(0)=Id.$ Hence we can write $\beta(X) = \prod_t \gamma_t se_{i_t j_t}(X f_t(X)) \gamma_t^{-1}$, where $se_{i_t j_t}(X f_t(X)) \in \ESp_{2n+2}^1(I_a[X])$, and $\gamma_t \in \ESp_{2n+2}^1(I_a)$. Using Lemma \ref{commESp(I)} we get that $\beta(Y^{4^r}X) = \prod_k se_{i_k j_k}(Y h_k(X,Y)/ a^m)$, with either $i_k =1$ or $j_k=1$, and $h_k(X,Y) \in I[X,Y]$. Note that  
\begin{eqnarray*}
se_{12}(Y h_k(X,Y)/ a^m) &=& I_{2n+2} + (Y h_k(X,Y)/ a^m) ~ e_1 ~ e_1^t ~ \psi_{n+1}, \\
se_{1 j_k}(Y h_k(X,Y)/ a^m) &=& I_{2n+2} + (Y h_k(X,Y)/ a^m) ~ e_1 ~ e_{\sigma(j_k)}^t ~ \psi_{n+1}\\
&& + (-1)^{j_k}(Y h_k(X,Y)/ a^m) ~ e_{\sigma(j_k)} ~ e_1^t ~ \psi_{n+1} ~ {\rm for} ~ j_k \ge 3, \\
se_{21}(Y h_k(X,Y)/ a^m) &=& I_{2n+2} + (Y h_k(X,Y)/ a^m) ~ e_2 ~ e_2^t ~ \psi_{n+1}, \\
se_{i_k 1}(Y h_k(X,Y)/ a^m) &=& I_{2n+2} + Y h_k(X,Y)/ a^m) ~ e_{i_k} ~ e_2^t ~ \psi_{n+1} \\
&& + (-1)^{i_k} Y h_k(X,Y)/ a^m) ~ e_2 ~ e_ {i_k}^t ~ \psi_{n+1}, ~ {\rm for} ~ i_k \ge 3. \\
\end{eqnarray*}

Let $\varepsilon_1, \ldots, \varepsilon_{2n}$ be the columns of the matrix $(1 \perp \varepsilon)^{-1} \in \E_{2n}(R_a)$. Let $\widetilde{e_i}$ denote the column vector $(I_3 \perp \varepsilon)^{-1} e_i$ of length $2n+2$. Note that $\widetilde{e_1}=e_1, \widetilde{e_2}=e_2$, and $\widetilde{e_i}^t = (0, 0, \varepsilon_{i-2}^t)$, for $i \ge 3$. Using Lemma \ref{phi=phi*,true-relative}  we can write $\alpha(Y^{4^r}X)_a$ as the product of elements of the form
\begin{eqnarray*}
&I_{2n+2} + (Y h_k(X,Y)/ a^m)  \widetilde{e_1} \widetilde{e_1}^t  \left( \begin{smallmatrix} \psi_1 & 0 \\ 0 & \varphi \end{smallmatrix} \right) = \mu_{\varphi}(0, Y h_k(X,Y)/ a^m), &\\
& I_{2n+2} + (Y h_k(X,Y)/ a^m) \widetilde{e_1} \widetilde{e_{\sigma(j_k)}}^t \left( \begin{smallmatrix} \psi_1 & 0 \\ 0 & \varphi \end{smallmatrix} \right )+ (-1)^{j_k} (Y h_k(X,Y)/ a^m) \widetilde{e_{\sigma(j_k)}} \widetilde{e_1}^1 \left( \begin{smallmatrix} \psi_1 & 0 \\ 0 & \varphi \end{smallmatrix} \right)& \\
&= \mu_{\varphi}(-(Y h_k(X,Y)/ a^m) \varepsilon_{j_k -2}, 0), & \\
&I_{2n+2} + (Y h_k(X,Y)/ a^m) \widetilde{e_2} \widetilde{e_2}^t \left( \begin{smallmatrix} \psi_1 & 0 \\ 0 & \varphi \end{smallmatrix} \right)= \rho_{\varphi}(0, -Y h_k(X,Y)/ a^m), &\\
& I_{2n+2} + (Y h_k(X,Y)/ a^m) \widetilde{e_{i_k}} \widetilde{e_2}^t \left( \begin{smallmatrix} \psi_1 & 0 \\ 0 & \varphi \end{smallmatrix} \right) & \\
&+ (-1)^{i_k} (Y h_k(X,Y)/ a^m) \widetilde{e_2} \widetilde{e_{i_k}}^t \left( \begin{smallmatrix} \psi_1 & 0 \\ 0 & \varphi \end{smallmatrix} \right) = \rho_{\varphi}(-(Y h_k(X,Y)/ a^m) \varepsilon_{\sigma(i_k)-2}, 0), & 
\end{eqnarray*}
 for $i_k, j_k \ge 3$. Let $s \ge 0$ be an integer such that $\widetilde{\varepsilon_i} = a^s \varepsilon_i \in P$ for all $i = 1, \ldots, 2n$. Let $d' = s+m$. Therefore, $\alpha((a^{d'}Y)^{4^r} X)_a$ is the product of elements of the form
 \begin{eqnarray*}
 & \rho_{\varphi}(0, -a^{d'} Y h_k(X, a^{d'} Y)/ a^m), \rho_{\varphi}(-(a^{d'} Y h_k(X, a^{d'} Y)/ a^m) \varepsilon_{j_k -2}, 0), & \\
 & \mu_{\varphi}(0, a^{d'} Y h_k(X, a^{d'} Y)/ a^m), ~{\rm and}~ \mu_{\varphi}(-(a^{d'} Y h_k(X, a^{d'} Y)/ a^m) \varepsilon_{\sigma(i_k)-2}). & 
 \end{eqnarray*}
 
Substituting $Y=1$ we get that $\alpha(a^dX)_a$ is the product of elements of the forms 
\begin{eqnarray*}
& \rho_{\varphi}(0, a^s h'_k(X)), \rho_{\varphi}(a^s h'_k(X) \varepsilon_{j_k-2}, 0), & \\
& \mu_{\varphi}(0, a^s h'_k(X)), ~{\rm and}~ \mu_{\varphi}(a^s h'_k(X) \varepsilon_{\sigma(i_k)-2}, 0), &
\end{eqnarray*}
where $h'_k(X) \in I[X]$.

Let us set $\alpha^*(X)$ to be the product of elements of the forms 
\begin{eqnarray*}
& \rho (0, a^s h'_k(X)), \rho (h'_k(X) \widetilde{\varepsilon_{j_k-2}}, 0), & \\
& \mu (0, a^s h'_k(X)), ~{\rm and}~  \mu (h'_k(X) \widetilde{\varepsilon_{\sigma(i_k)-2}}, 0). &
\end{eqnarray*}

Note that $\alpha^*(X)$ belongs to $\ETrans_{\Sp}(Q[X], \langle, \rangle)$. It is clear from the construction that $\alpha^*(0)=Id.$ and  $\alpha^*(X)$ localises to $\alpha(bX)$, for some $b \in (a^d), d \gg 0$.
\end{proof}

Next we state a Local-Global principle for $\ETrans_{\Sp}(IQ,\langle , \rangle)$.

\begin{lem} \label{symp,LG,true-rel} Let $R$ be a commutative ring with
  $R=2R$, and let $I$ be an ideal of $R$. Let $(P,\langle , \rangle)$
  be a symplectic $R$-module with $P$ finitely generated projective
  module of rank $2n$, $n \ge 1$ and $Q=(R^2 \oplus P)$ with the induced form
  on $(\mathbb{H}(R)~\oplus~P)$. Let $\alpha(X) \in {\Sp}(Q[X])$, with $\alpha(0) = Id$. 
  If for each maximal ideal $\gm$ of $R$, $\alpha(X)_\gm \in
  \ETrans_{\Sp}(IQ_\gm[X],\langle , \rangle_{\psi_1 \perp
    \varphi_\gm})$,
  then $\alpha(X) \in \ETrans_{\Sp}(IQ[X],\langle , \rangle)$.
\end{lem}

\begin{proof}
Proof of Lemma \ref{linear,LG,true-rel} works verbatim for this.
\end{proof}

Before we prove the main result (Theorem \ref{Tits-sympTrans} ) of this section, 
we state a lemma to show that symplectic transvections are homotopic to identity.

\begin{lem} \label{symp-trans-hom-to-identity}
Let $(P, \langle, \rangle)$ be a symplectic $R$-module and $\alpha \in
\Trans_{\Sp}(P, \langle, \rangle)$. Then there exists $\beta(X) \in
\Trans_{\Sp}(P[X], \langle, \rangle)$ such that $\beta(1)=\alpha$ and
$\beta(0)=Id.$
\end{lem}

\begin{proof}
See proof of Lemma 5.22 in \cite{cr2}.
\end{proof}

\begin{thm} \label{Tits-sympTrans} Let $R$ be a commutative ring with $R=2R$,
  and let $I$ be an ideal of $R$. Let $(P, \langle , \rangle)$ be a
  symplectic $R$-module with $P$ a finitely generated projective module
  of rank $2n$, $n \ge 2$ and $Q=(R^2 \oplus P)$ with the induced form
  on $(\mathbb{H}(R)~\oplus~P)$. Then $\ETrans_{\Sp}(Q,I^4Q) \subseteq \ETrans_{\Sp}(IQ)$.
\end{thm}

\begin{proof}
Let us consider an element $\gamma \in \ETrans_{\Sp}(Q,I^4Q)$. Note 
that $\gamma$ is product of elements of the form $g_1 ~ \rho(q_1, \alpha) ~ g_1^{-1}$
and $g_2 ~ \mu(q_2, \beta) ~ g_2^{-1}$, where $g_1, g_2 \in \ETrans_{\Sp}(Q, \langle, \rangle)$, $q_1, 
q_2 \in I^4P$, and $\alpha, \beta \in I^4$. Let us set $\gamma(X)$ to be the
product of elements of the form $g_1(X) ~ \rho(q_1 X^4, \alpha X^4) ~ g_1(X)^{-1}$
and $g_2(X) ~ \mu(q_2 X^4, \beta X^4) ~ g_2(X)^{-1}$, where $g_1(X), g_2(X) \in \ETrans_{\Sp}(Q[X],
\langle, \rangle)$. Here $\rho(q_1 X^4, \alpha X^4), \mu(q_2 X^4, \beta X^4) \in \ETrans_{\Sp}(I^4Q[X^4], \langle, \rangle)$,
with $q_1X^4, q_2 X^4 \in I^4 P[X^4]$ and $\alpha X^4, \beta X^4 \in I^4[X^4]$. Note that over a local ring
  $R_\gm$, for any maximal ideal $\gm$ of $R$, the alternating form
  $\langle , \rangle$ corresponds to the alternating matrix
  $\varphi_\gm$ with respect to some basis. Localizing
at a maximal ideal $\gm$ of $R$, we get that $\gamma(X)_\gm$ is product
of elements of the form $g_1(X)_\gm ~\rho(q_1 X^4, \alpha X^4)_\gm ~ g_1(X)_\gm ^{-1}$
and $g_2(X)_\gm ~\mu(q_2 X^4, \beta X^4)_\gm~ g_2(X)_\gm ^{-1}$. These 
elements belong to $\ETrans_{\Sp}(Q_\gm[X],I^4Q_\gm[X^4], \langle, \rangle_{\psi_1 \perp \varphi_\gm})$. By Lemma \ref{local,rel}, we have $\varphi_\gm = (1 \perp \varepsilon(\gm))^t \psi_n (1 \perp \varepsilon(\gm))$ for some $\varepsilon(\gm) \in \E_{2n-1}(R_\gm)$. Therefore, by Lemma \ref{phi=phi*,true-relative} we have 
\begin{eqnarray*}
&&\ETrans_{\Sp}(Q_\gm[X], I^4Q_\gm[X^4], \langle , \rangle_{\psi_1 \perp \varphi_\gm}) \\
&=& (I_3 \perp \varepsilon(\gm))^{-1}~ \ETrans_{\Sp}(Q_\gm[X], I^4Q_\gm[X^4],\langle , \rangle_{\psi_{n+1}}) ~ (I_3 \perp \varepsilon(\gm)).
\end{eqnarray*}

Note that the group $\ETrans_{\Sp}(Q_\gm[X], I^4Q_\gm[X^4],\langle , \rangle_{\psi_{n+1}})$ is equal to the group $\ESp_{2n+2}(R_\gm[X], (I_\gm[X])^4)$ by Lemma \ref{symp,free,rel} . Using Theorem \ref{tits-symp} we get that
$\gamma(X)_\gm \in (I_3 \perp \varepsilon(\gm))^{-1} \ESp_{2n+2}((I_\gm[X])^2) (I_3 \perp \varepsilon(\gm)) \subseteq (I_3 \perp \varepsilon(\gm))^{-1}\ESp_{2n+2}^1(I_\gm[X]) (I_3 \perp \varepsilon(\gm))$
(see Remark \ref{inclusion-symplectic}). Note that $\ESp_{2n+2}^1(I_\gm[X]) = \ETrans_{\Sp}(IQ_\gm[X], \langle, \rangle_{\psi_{n+1}})$ by Lemma \ref{symp,free,true-rel}. Therefore, $\gamma(X)_\gm \in \ETrans_{\Sp}(IQ_\gm[X], \langle, \rangle_{\psi_1 \perp \varphi_\gm})$ by Lemma \ref{phi=phi*,true-relative}.
Using Lemma \ref{symp,LG,true-rel}, we get that $\gamma(X) \in
\ETrans_{\Sp}(IQ[X], \langle, \rangle)$. Substituting $X=1$ we get that $\gamma \in
\ETrans_{\Sp}(IQ, \langle, \rangle)$.
\end{proof}

\begin{rmk}
{\rm We conclude this section with a remark similar to the Remark \ref{rmk-linTrans} in the previous section.
As before recall that Tits's result for the elementary symplectic group says that $\ESp_{2n}(R, I^2) \subseteq \ESp_{2n}(I)$ (see Theorem \ref{tits-symp}). However, our theorem as above involves fourth power of an ideal. In view of Tits's result 
it is desirable to establish the above theorem for
square of an ideal.}
\end{rmk}

\noindent
{\bf Acknowledgement:} The author thanks B. Sury for bringing to her notice the paper \cite{N} of B. Nica which 
initiated this work and for going through an earlier version of the preprint. The author 
also thanks Department of Science and Technology, Govt. of India for INSPIRE Faculty Award 
[IFA-13 MA-24] that supported this work.


\end{document}